\documentclass[10pt]{amsart}

\usepackage{amssymb}

\usepackage{enumerate}

\usepackage{helvet} 

\usepackage{courier} 

\usepackage{eucal} 

\usepackage{mathrsfs} 

\reversemarginpar 
\normalmarginpar 

\let\oldmarginpar\marginpar 
\renewcommand\marginpar[1]{\-\oldmarginpar{\raggedright\small\sf #1}}

\title{The algebraic dynamics of generic endomorphisms of $\P^n$}

\author{Najmuddin Fakhruddin}
 
\address{School of Mathematics, Tata Institute of Fundamental Research, 
Homi Bhabha Road, Mumbai 400005, India}

\email{naf@math.tifr.res.in}

\newcommand{\nc}{\newcommand}

\nc{\rnc}{\renewcommand}

\nc{\bs}{\backslash}
\nc{\te}{\otimes}
\nc{\lf}{\lfloor} 
\nc{\rf}{\rfloor}
\nc{\lc}{\lceil}  
\nc{\rc}{\rceil}
\nc{\lr}{\longrightarrow}
\nc{\sr}{\stackrel}
\nc{\dar}{\dashrightarrow}
\nc{\thra}{\twoheadrightarrow}

\nc{\mc}{\mathscr} 
\nc{\mb}{\mathbb}
\nc{\mf}{\mathbf}
\nc{\mr}{\mathrm}
\nc{\mg}{\mathfrak}

\nc{\bP}{\mathbb{P}}
\rnc{\P}{\mathbb{P}}
\nc{\Q}{\mathbb{Q}}
\nc{\Z}{\mathbb{Z}}
\nc{\C}{\mathbb{C}}
\nc{\R}{\mathbb{R}}
\nc{\A}{\mathbb{A}}
\nc{\V}{\mathbb{V}}
\nc{\W}{\mathbb{W}}
\nc{\N}{\mathbb{N}}
\nc{\F}{\mathbb{F}}
\nc{\G}{\mathbb{G}}
\nc{\qb}{\overline{\mathbb{Q}}}

\nc{\aff}{{\A}^1}
\nc{\naive}{\!\sim_n}
\nc{\Spec}{\mr{Spec}}

\nc{\ep}{\epsilon}
\nc{\ve}{\varepsilon}
\rnc{\l}{\lambda}

\nc{\wt}{\widetilde}
\nc{\wh}{\widehat}
\nc{\ov}{\overline}

\nc{\res}{\operatorname{Res}}
\nc{\Mor}{\operatorname{Mor}}
\nc{\Per}{\operatorname{Per}}
\nc{\prep}{\operatorname{Prep}}
\nc{\End}{\operatorname{End}}
\nc{\Orb}{\operatorname{Orb}}
\nc{\Gal}{\operatorname{Gal}}
\nc{\ns}{\operatorname{NS}}

\nc{\lcm}{\operatorname{lcm}}
\nc{\ch}{\operatorname{char}}

\newtheorem{thm}{Theorem}[section]
\newtheorem{prop}[thm]{Proposition}

\newtheorem{cor}[thm]{Corollary}
\newtheorem{lem}[thm]{Lemma}

\theoremstyle{definition}
\newtheorem{defn}[thm]{Definition}

\newtheorem{rem}[thm]{Remark}

\numberwithin{equation}{section}

\begin{document}

\begin{abstract}
  We investigate some general questions in algebraic dynamics in the
  case of generic endomorphisms of projective spaces over a field of
  characteristic zero. The main results that we prove are that a
  generic endomorphism has no non-trivial preperiodic subvarieties,
  any infinite set of preperiodic points is Zariski dense and any
  infinite subset of a single orbit is also Zariski dense, thereby
  verifying the dynamical ``Manin--Mumford'' conjecture of Zhang and
  the dynamical ``Mordell--Lang'' conjecture of Denis and
  Ghioca--Tucker in this case.
\end{abstract}

\maketitle

\section{Introduction}

The goal of this article is to study some aspects of the algebraic
dynamics of generic endomorphisms\footnote{The precise meaning of
  generic endomorphism is given in Definition \ref{def:generic} but we
  note here that when $K = \C$ this means we consider endomorphisms in
  the complement of a countable union of proper subvarieties of the
  natural parameter variety of endomorphisms of degree $d$.}  of
$\P^n$ of degree $d>1$ over a field $K$ of characteristic zero.
Properties of algebraic varieties, for example smooth projective
curves or abelian varieties, are often easier to derive for generic
varieties than for arbitrary varieties, the main reason being that one
has a great deal of freedom in choosing specialisations. It is natural
to expect that the same holds for algebraic dynamical systems; we show
that this is indeed the case for generic endomorphisms of $\P^n$. We
prove three results for such endomorphisms, two of which have
analogues which are expected to hold much more generally but are
presently far from being known.

Our main result is
\begin{thm} \label{thm:igt} 
  Let $f: \P^n_K \to \P^n_K$ be a generic endomorphism of degree $d >
  1$ over an algebraically closed field $K$ of characteristic zero.
  For each $x \in \P^n(K)$, every infinite subset of $O_f(x)$, the
  $f$-orbit of $x$, is Zariski dense in $\P^n_K$.
\end{thm}
This implies the dynamical ``Mordell--Lang'' conjecture of Denis
\cite{denis-suites} and Ghioca--Tucker \cite{ghioca-tucker-conj} for
generic endomorphisms. This conjecture has been proved for etale
endomorphisms of arbitrary varieties by Bell, Ghioca and Tucker
\cite{bell-ghioca-tucker} but there are only a few other cases where
it is known.  The proof of this theorem is based on two other
results. The first is
\begin{thm} \label{thm:iinv} 
  Let $f:\P^n_K \to \P^n_K$ be a generic endomorphism of degree $d >
  1$ over an algebraically closed field $K$ of characteristic zero. If
  $X \subset \P^n_K$ is an irreducible subvariety such that $f^r(X) =
  X$ for some $r > 0$, then $X$ is a point or $X = \P^n_K$.
\end{thm}
This is a rather straightforward consequence of the transitivity of
the monodromy action on the set of periodic points of a fixed period of a
generic endomorphism, which we prove (Proposition \ref{prop:galois})
using a result of Bousch \cite{bousch-thesis}, Lau--Schleicher
\cite{lau-schleicher} and Morton \cite{morton} for polynomials in one
variable of the form $z \mapsto z^d + c$.

\smallskip

We then extend Proposition \ref{prop:galois} to prove transitivity of
the monodromy action on the set of \emph{preperiodic} points of fixed
period and preperiod of a generic endomorphism (Proposition
\ref{prop:prep}); this does not hold for the $1$-parameter family of
polynomials mentioned above and we use a $2$-parameter family
containing this.  This allows us to prove Zhang's ``Manin--Mumford''
conjecture \cite{ghioca-tucker-zhang} in the case of generic
endomorphisms in the following strong form.

\begin{thm} \label{thm:izhang} 
  For $f: \P^n_K \to \P^n_K$ a generic endomorphism of degree $d > 1$ over
  an algebraically closed field $K$ of characteristic zero, any
  infinite subset of $\P^n(K)$ consisting of $f$-preperiodic points is
  Zariski dense in $\P^n_K$.
\end{thm}

We prove Theorem \ref{thm:igt} by combining Theorems \ref{thm:iinv}
and \ref{thm:izhang} with some $p$-adic as well as mod $p$ arguments.
Note that the statement does not invove (pre)periodic points in any
way. However, using a lifting argument for periodic points we show
that any subvariety $Y$ containing an infinite subset of $O_f(x)$ must
contain infinitely many periodic points, or $x$ can be specialised in
such a way that one may apply the $p$-adic interpolation argument used
by Bell, Ghioca and Tucker in their proof of the conjecture for etale
endomorphisms in \cite{bell-ghioca-tucker}.  Theorem \ref{thm:izhang}
and Theorem \ref{thm:iinv} then force $Y$ to be equal to $\P^n_K$ in
either of these cases.

\bigskip

\noindent {\bf Acknowledgements.}
I thank Ekaterina Amerik, Amitava Bhattacharya, Bjorn Poonen, Joseph
Silverman and Thomas Tucker for useful conversations and comments. I
also thank the referee for his/her suggestions for improvements in the
exposition.

\section{Preliminaries}

\subsection{}
Let $X$ be a set and $f:X \to X$ any map. By $f^n$ we shall mean the
$n$-fold composite of $f$ with itself. For $x \in X$ we denote by
$O_f(x)$ its orbit under $f$, \emph{i.e.}, the set $\{f^n(x) \}_{n \geq
  0}$. A point $x \in X$ is said to be $f$-\emph{periodic} if $f^n(x)
= x$ for some $n>0$. The smallest such integer is called the
\emph{period} of $x$. We denote the set of all periodic points of
period $b$ by $\Per_f(b)$. A point $x \in X$ is said to
$f$-\emph{preperiodic} if $O_f(x)$ is finite.  The \emph{preperiod} of
$f$ is the smallest non-negative integer $a$ such that $f^a(x)$ is
periodic and the \emph{period} of $x$ is the period of any periodic
point in its orbit.  We denote by $\prep_f(a,b)$ the set of all such
points. Let $\Orb_f(b)$ denote the set of orbits of $f$-periodic
points of period $b$. If this is finite then $|\Per_f(b)| = b \cdot
|\Orb_f(b)|$.  We drop $f$ from any of the notation introduced above
if there is no scope for confusion.

If $X$ is an algebraic variety over a field $K$ and $f:X \to X$ is a
morphism over $K$, we use the same notation as above for the induced
map on the set of $L$-rational points of $X$ for any extension field
$L$ of $K$.

\subsection{} \label{subsec:mono}

Let $S$ be a smooth irreducible variety over a field $k$ and let $g:Z
\to S$ be a finite flat morphism. By the \emph{monodromy} or
\emph{Galois action} of $g$ we shall mean the action of
$\Gal(\ov{k(S)}/k(S))$ on (the points of) a geometric generic fibre of
$g$. If $g$ is generically smooth---this is always true if $\ch(k)=0$
and $Z$ is reduced---there is a Zariski open subset $U$ of $S$ such
that $g$ induces a finite etale morphism $g^{-1}(U) \to U$ and then
the monodromy may be interpreted as an action of $\pi_1^{et}(U, *)$,
where $*$ is a geometric point of $U$. If $k = \C$, it may be
interpreted as an action of $\pi_1^{top}(U, *)$.

The monodromy action is transitive iff $Z$ is irreducible. If $Z$ is
generically smooth, this is equivalent to $Z^0$, the smooth locus of
$Z$, being connected or, if $k = \C$, path connected.

\subsection{} \label{subsec:p}

\begin{defn} \label{def:qp} 
  Let $\pi:\mc{X} \to S$ be a projective morphism and $f:\mc{X} \to
  \mc{X}$ a surjective morphism over $S$. We say that $f$ is
  \emph{quasi-polarised} if there exists a line bundle $\mc{L}$ on
  $\mc{X}$ such that $f^*(\mc{L}) \otimes \mc{L}^{-1}$ is $S$-ample.
\end{defn}

For any morphism $\pi:\mc{X} \to S$ and $f:\mc{X} \to \mc{X}$ a
morphism over $S$, we denote by $\Gamma_f$ the graph of $f$ in $\mc{X}
\times_S \mc{X}$.  Let $P_f(n)$ be the closed subscheme of $\mc{X}$
defined by the intersection of $\Gamma_{f^n}$ with the diagonal. A
geometric point of the fibre of $P_f(n)$ over any point $s \in S$ is a
periodic point of period dividing $n$ of the map $f_s$ of $\mc{X}_s$
induced by $f$. Similarly, let $P_f(m,n)$ be the intersection of
$\Gamma_{f^{m}}$ and $\Gamma_{f^{n}}$ which we view as a subscheme of
$\mc{X}$ via the first projection.

\begin{lem} \label{lem:lift} 
  Let $\pi:\mc{X} \to S$ be a smooth projective morphism with $S$ a
  regular irreducible finite dimensional scheme and let $f:\mc{X} \to
  \mc{X}$ be a finite quasi-polarised morphism. Then
\begin{enumerate}
\item For any $m,n \geq 0$, $m \neq n$, $P_f(m,n)$ is finite and flat
  over $S$.
\item For any $s_1, s_2$ in $S$ with $s_2$ a specialisation of $s_1$,
  any element of $\prep_{f_{s_2}}(a,b)$ can be lifted to an element of
  $\prep_{f_{s_1}}(a,b)$.
\end{enumerate}
\end{lem}

\begin{proof}
  Since $f$ is proper, $P_f(m,n)$, being a closed subscheme of
  $\mc{X}$, is also proper over $S$.  The dimension of each
  irreducible component of $P_f(m,n)$ is at least equal to $\dim(S)$
  since the codimension of $\Gamma_{f^i}$ in $\mc{X} \times_S \mc{X}$
  is the relative dimension of $\mc{X}$ over $S$. To prove $P_f(m,n)$
  is finite over $S$ it suffices to show that the fibres of this map
  are finite since proper quasi-finite morphisms are
  finite. Furthermore, the finiteness of the fibres implies that the
  dimension of each component of $P_f(m,n)$ is exactly $\dim(S)$,
  hence $P_f(m,n)$ is a local complete intersection in $\mc{X}
  \times_S \mc{X}$) which is regular, which implies that $P_f(m,n)$ is
  Cohen--Macaulay.  Since the dimension of each irreducible component
  of $P_f(m,n)$ is at least $\dim(S)$, each such component dominates
  $S$ if all the fibres are finite. It then follows from the fibrewise
  flatness criterion that all of (1) is a consequence of the
  finiteness of the fibres.

  Let $\mc{L}$ be a line bundle on $\mc{X}$ so that $M = f^*(\mc{L})
  \otimes \mc{L}^{-1}$ is ample. Then $(f^m)^*(\mc{L}) = \mc{L}
  \otimes M \otimes f^*(M) \otimes \dots (f^{m-1})^* (M)$ and
  similarly for $n$. By the construction of $P_f(m,n)$,
  $(f^m)^*(\mc{L})$ and $(f^m)^*(\mc{L})$ restrict to the same line
  bundle on it, so assuming $m >n$ without loss of generality, we get
  that $(f^{m-1})^*(M)\otimes \dots \otimes (f^n)^*(M)$ is trivial on
  $P_f(m,n)$. But $M$ is ample, hence so is $(f^i)^*(M)$ for all $i
  \geq 0$ since $f$ is finite. There is at least one factor in the
  above tensor product of line bundles so this is only possible if all
  the fibres are finite.

  If $x \in \prep_{f_{s_2}}(a,b)$, then $x$ occurs in the fibre of
  $P_f(a+b,a)$ over $s_1$ and does not occur in the fibre of any
  $P_f(m,n)$ for $m<a+b$ or $n< a$. By (1), there is a point
  $\tilde{x}$ in the fibre of $P_f(a+b,a)$ specialising to $x$. By the
  definition of $P_f(m,n)$, it follows that $f_{s_1}^{a+b}(\tilde{x})
  = f_{s_1}^{a}(\tilde{x})$, so $\wt{x}$ is preperiodic with preperiod
  $\leq a$ and period $\leq b$. Since neither the preperiod nor the
  period can increase under specialisation, and $\wt{x}$ specializes
  to $x$, (2) follows.
\end{proof}

\section{Periodic points and periodic subvarieties}

\subsection{}

Let $\Mor_{n,d}$ be the scheme over $\Z$ representing morphisms of
$\P^n_{\Z}$ to itself of algebraic degree $d$. Its $k$-valued points,
for any field $k$, consist of $n+1$-tuples of homogenous polynomials
of degree $d$ over $k$ without common zeros in $\P^n_k$, up to a
scalar. It is smooth and of finite type over $\Z$ and has
geometrically irreducible fibres. For any field $L$, we denote
$\Mor_{n,d} \times_{\Spec(\Z)} \Spec(L)$ by $\Mor_{n,d/L}$.

\begin{defn} \label{def:generic} 
  If $K$ is an algebraically closed field, we say that an endomorphism
  $f$ of $\P^n_K$ is \emph{generic} if the image of the induced map
  $\Spec(K) \to \Mor_{n,d}$ corresponding to some conjugate of $f$ by
  an element of $PGL_{n+1}(K) = \mr{Aut}(\P^n_K)$ is the generic point
  of a fibre of the structure morphism $\Mor_{n,d} \to \Spec(\Z)$.
\end{defn}

If $K= \C$, the set of points in $\Mor_{n,d}(\C)$ corresponding to
generic morphisms is the complement of a countable union of proper
subvarieties.

\subsection{}

We recall the theorem of Bousch \cite[Chapitre 3, Th\'eor\`eme
4]{bousch-thesis}, Lau--Schleicher \cite[Theorem 4.1]{lau-schleicher}
and Morton \cite[Theorem D]{morton} alluded to earlier; the statement
below is \cite[Theorem 10]{morton}, except that we have replaced the
field $\ov{\Q}$ there by $k$.

\begin{thm} \label{thm:morton} 
  Let $k$ be a field of characteristic zero and let $f(z) = z^d + t$
  with $t$ transcendental over $k$ and $d \geq 2$. For any $b \geq 1$,
  the Galois group of the polynomial $f^b(z) - z$ over $k(t)$ is the
  direct product $\prod_{e \mid b} (\Z/e\Z \ \mr{wr} \ S_{r_e})$, where
  $e \cdot r_e$ is the number of periodic points of period $e$ over
  $\ov{k(t)}$ and $\mr{wr}$ denotes the wreath product.
\end{thm}

The theorem can be interpreted as saying that the Galois action is as
large as possible given that it must commute with the action of
$f$. One may expect that a similar result holds for generic
endomorphisms of projective spaces of arbitrary dimension; however,
the following proposition, for the proof of which we will use only the
transitivity of the Galois action in the above theorem, suffices for
our applications.

\begin{prop} \label{prop:galois} 
  Let $k$ be a field of characteristic zero and let $k_{n,d}$ be the
  function field of $\Mor_{n,d/k}$. Let $f_{n,d}$ be the endomorphism
  of $\P^n_{k_{n,d}}$ corresponding to the generic point of
  $\Mor_{n,d/k}$ and let $b$ be any positive integer. Then
  $\Gal(\ov{k_{n,d}}/k_{n,d})$ acts transitively on
  $\Per_{f_{n,d}}(b)$.
\end{prop}

\subsection{}

For a field $k$ and any element $\l \in k$, let $\phi_{\l}: \A^1_k \to
\A^1_k$ be the map given by $z \mapsto z^d + \l$; the integer $d$ will
be assumed to be fixed whenever we use this notation.  The periodic
points of $\phi_0$ are $0$ and the roots of unity of order prime to
$d$: if $\zeta$ is a primitive $n$-th root of unity with $(n,d) = 1$
then the period of $\zeta$ is the order of $d$ in $(\Z/n\Z)^{\times}$.

We shall need the following simple lemma for the proof of Proposition
\ref{prop:galois}.
\begin{lem} \label{lem:roots} 
  Let $d >1$ and $m,m' \geq 1$ be integers such that $(m,d) = (m',d)
  =1$. Assume that the highest powers of $2$ dividing $m$ and $m'$ are
  unequal or are both equal to $1$. Let $a$ (resp.~$a'$) be the order
  of $d$ in $(\Z/m\Z)^{\times}$ (resp.~$(\Z/m'\Z)^{\times}$).
\begin{enumerate}
\item The order of $d$ in $(\Z/\lcm(m,m')\Z)^{\times}$ is divisible by
  $\lcm(a,a')$.
\item There exists a root of unity $\zeta$ (resp.~$\zeta'$) of order
  $m$ (resp.~$m'$) so that $\zeta{\zeta'}^{-1}$ is of order
  $\lcm(m,m')$.
\item For $\zeta$, $\zeta'$ as above, there exists a primitive
  $\lcm(m,m')$-th root of unity $\eta$ so that $\eta \zeta$ is a
  primitive $\lcm(m,m')$-th root of unity.
\end{enumerate}
\end{lem}

\begin{proof}
  The natural quotient map from $\Z/\lcm(m,m')\Z$ to $\Z/m\Z$
  (resp.~$\Z/m'\Z$) induces a group homomorphism from
  $(\Z/\lcm(m,m')\Z)^{\times}$ to $(\Z/m\Z)^{\times}$
  (resp.~$(\Z/m'\Z)^{\times}$). This implies that the order of $d$ in
  $(\Z/\lcm(m,m')\Z)^{\times}$ is divisible by $\lcm(a,a')$.

  To prove (2) and (3), we may reduce to the case that $m$ and $m'$
  are powers of the same prime $p$. Let $P \subset \Z/p^r\Z$ be the
  set of generators, so $|P| = p^r - p^{r-1}$. If $p > 2$, then $|P|>
  p^{r-1}$ so the translate of $P$ by any element of $\Z/p^r\Z$ has a
  non-empty intersection with $P$. If $p=2$, the claim follows from
  the extra condition since the translate of $P$ by an element not in
  $P$ always intersects $P$ non-trivially.
\end{proof}

\begin{proof}[Proof of Proposition \ref{prop:galois}]
  If $n=1$ and $b>1$ the proposition follows immediately from Theorem
  \ref{thm:morton}. If $b=1$ a much simpler version of the argument
  below shows transitivity; since we do not use this later we leave
  the details to the reader.

  We now assume $n>1$.  Consider the morphism $g_{n,d}:\P^n_k \to
  \P^n_k$ given by $[x_0,\dots,x_n] \mapsto [x^d_0,\dots,x^d_n]$. The
  set $\Per_{g_{n,d}}(b)$ consists of points which have a
  representative $[\xi_0,\xi_1,\dots,\xi_n]$ with each $\xi_i$ equal
  to $0$ or a $d^b-1$-th root of unity. The standard affine charts of
  $\P^n_k$ given by the locus where a fixed coordinate is non-zero are
  preserved by $g_{n,d}$. A simple computation on each such chart
  shows that the eigenvalues of the differential of $g_{n,d}^b$ at a
  fixed point are equal to $d^b \xi$, where $\xi$ is a root of unity
  or $0$. This is never equal to $1$ since $d>1$, so
  $\Gamma_{g_{n,d}^b}$ and the diagonal intersect transversely in
  $\P^n_k \times \P^n_k$ for all $b>0$. Consequently, all periodic
  points of $g_{n,d}$ have multiplicity one so we may use $g_{n,d}$ as
  a basepoint in $\Mor_{n,d}$ in order to compute the Galois action on
  $\Per_{f_{n,d}}$.

\smallskip

  For $0\leq i \leq n$, consider the family of endomorphisms $f_i:
  \P^n_k \times \A^n_k \to \P^n_k \times \A^n_k$ given by
\begin{multline*}
f_i(([x_0,\dots,x_i,\dots,x_n], (c_1,c_2,\dots,c_n))) = \\
([x_0^d +
c_1x_i^d,\dots, x_{i-1}^d + c_ix_i^d,x_i^d, x_{i+1}^d +
c_{i+1}x_i^d,\dots,x_n^d + c_nx_i^d],(c_1,c_2,\dots,c_n)) \ .
\end{multline*}
On the open affine $U_i$ given by $x_i \neq 0$, $f_i$ is the product
of the $n$ polynomials $\phi_{c_j}$. On the complement of this affine,
\emph{i.e.}, on the subvariety given by $x_i = 0$ (which is also
preserved by $f_i$), the maps do not depend on $c_j$, so the monodromy
action of this family on the periodic points in this locus is trivial.
Let $G_i$ be the subgroup of the monodromy group acting on
$\Per_{g_{n,d}}$ corresponding to this family; by applying Theorem
\ref{thm:morton} one gets a complete description of this group. We let
$G$ be the subgroup of the monodromy group generated by all the $G_i$.

Let $P = [\zeta,1,\dots,1,1]$ where $\zeta$ is in $\Per_{\phi_0}(b)$
and let $Q = [\xi_0,\xi_1,\dots,\xi_n]$ be any other element of
$\Per_{g_{n,d}}(b)$. We may assume that some $\xi_i = 1$, so each
$\xi_j \in \Per_{\phi_0}(b')$ for some $ b' \mid b$, and also the lcm
of the periods of all the $\xi_j$ is $b$.  We prove the transitivity
of the monodromy action by showing that there exists an element in the
monodromy which sends $P$ to $Q$.

From the transitivity of the Galois action in Theorem \ref{thm:morton}
it follows that for $\xi_j \in \Per_{\phi_0}(b')$, as long as some
$\xi_i = 1$ with $i \neq j$, we can find an element of $G$ which fixes
all coordinates of $Q$ except that it replaces $\xi_j$ with any other
$\xi_j' \in \Per_{\phi_0}(b')$. Since $0, 1 \in \Per_{\phi_0}(0)$, we
may use this to assume that all $\xi_i \neq 0$ and then also, by
dividing through by $\xi_n$, that $\xi_n = 1$.

We now show that we may also assume that $\xi_0 \in \Per_{\phi_0}(b)$.
Suppose $\xi_0$ is a primitive $m$-th root of unity, $\xi_1$ is a
primitive $m'$-th root of unity and $\xi_0 \in \Per_{\phi_0}(a)$,
$\xi_1 \in \Per_{\phi_0}(a')$. By using the action of $G$ we may
change $\xi_0$ and $\xi_1$ so that that $m = d^a -1$, $m' = d^{a'} -
1$. If the highest powers of $2$ dividing $m$ and $m'$ are equal and
greater than $1$ we may change $\xi_0$ to a primitive $(d^a -1)/2$-th
root of unity; the period $a$ remains unchanged.  By Lemma
\ref{lem:roots} we may then assume that $\xi_0{\xi_1}^{-1}$ is a
primitive $\lcm(m,m')$-th root of unity. We multiply all coordinates
of $Q$ by $\xi_1^{-1}$ so the zeroth coordinate becomes
$\xi_0{\xi_1}^{-1}$ and the second becomes $1$.  Using the action of
$G$, we then replace the zeroth coordinate by $\eta$ as in Lemma
\ref{lem:roots} while keeping all other coordinates fixed. We then
multiply all coordinates by $\xi_1$. The resulting point has all
coordinates except for the zeroth equal to the corresponding
coordinates of $Q$ while the zeroth coordinate is now in
$\Per_{\phi_0}(\lcm(a,a'))$. Repeating this procedure with $\xi_1$
replaced by $\xi_2$, then $\xi_3$ and so on, since the lcm of the
periods of all the $\xi_i$ is $b$, it follows that eventually we have
that $\xi_0 \in \Per_{\phi_0}(b)$.

We now inductively transform $P$ into $Q$ using the action of $G$. If
$\xi_i$ has period $a_i$ as an element of $\Per_{\phi_0}$, we may use
the action of $G$ to replace it by a primitive $d^{a_i} -1$-th root of
unity. If the highest power of $2$ dividing $d^{a_i} -1$ is equal to
the highest power of $2$ dividing $d^b -1$ and $d$ is odd, then we use
a primitive $d^{a_i} -1/2$-th root of unity instead if $i>0$.  

Let $P_0 = P$ and suppose we have inductively constructed $P_i=
[\xi_{0,i},\xi_{1,i},\dots,\xi_{n-1,i},1] \in \Per_{\phi_0}(b)$ with
the following properties:
\begin{enumerate}
\item $\xi_{0,i} \in \Per_{\phi_0}(b)$
\item $\xi_{j,i} = \xi_j$ for $0<j\leq i$
\item $\xi_{j,1} = 1$ for $j>i$.
\end{enumerate}
$P_0$ clearly satisfies these properties; we will show that given
$P_i$ with $i < n$ we can find an element of $G$ which transforms it
into a point $P_{i+1}$ with the required properties.

So suppose $P_i$ has been constructed. By Lemma \ref{lem:roots} there
exists $\eta \in \Per_{\phi_0}(b)$ so that $\eta^{-1} \xi_{i+1} \in
\Per_{\phi_0}(b)$. Since $\xi_{0,i} \in \Per_{\phi_0}(b)$ we may use
the action of $G$ to replace $\xi_{0,i}$ with $\eta$ while keeping the
other coordinates fixed. We then multiply all coordinates by
$\eta^{-1}$ so the zeroth coordinate becomes $1$ and the $i+1$-th
coordinate becomes $\eta^{-1}$. Since $\eta^{-1} \xi_{i+1} \in
\Per_{\phi_0}(b)$ we may use the action of $G$ to replace the $i+1$-th
coordinate by $\eta^{-1} \xi_{i+1} $ while keeping all the other
coordinates fixed. If we now multiply all coordinates by $\eta$ we
obtain a point $P_{i+1}$ with the property that all coordinates except
the zeroth and $i+1$-th of $P_i$ and $P_{i+1}$ are equal, the zeroth
coordinate is $\eta \in \Per_{\phi_0}(b)$ and the $i+1$-th is
$\xi_{i+1}$, so fulfilling the requirements.

We thus obtain a point $P_{n}$ with the property that all coordinates
of $P_{n}$ and $Q$ are equal, except possibly for the first. Since the
zeroth coordinate of both is in $\Per_{\phi_0}(b)$ and the $n$'th is
$1$, we may use an element of $G$ to transform $P_n$ into $Q$. It
follows that the action of $G$, hence of the full monodromy group, is
transitive on $\Per_{\phi_0}(b)$.
\end{proof}

\begin{cor} \label{cor:ram}
  No preperiodic points of $f_{n,d}$ lies in the ramification
  locus.
\end{cor}

\begin{proof}
  The ramification locus of $f_{n,d}$ is defined over $k_{n,d}$, thus
  if one preperiodic point lies in the ramification locus, so must its
  entire Galois orbit.

  The Galois orbit must map onto the Galois orbit of the corresponding
  periodic point, \emph{i.e.}~the periodic point $y$ such that
  $f_{n,d}^r(x) = y$ and $f_{n,d}^s(x)$ is not periodic for any $s<r$.
  But this orbit consisits of all periodic points of a fixed period
  $b$ by Proposition \ref{prop:galois}. By specialisation to the
  $d$-power map $g_{n,d}$, we see that this is not possible: the
  ramification locus of this map is actually preserved by the map and
  for no period $b$ are all the periodic points of period $b$
  contained in the ramification locus.
\end{proof}

\subsection{}

\begin{thm} \label{thm:inv} 
  Let $f:\P^n_K \to \P^n_K$ be a generic endomorphism of degree $d >
  1$ over an algebraically closed field $K$ of characteristic
  zero. Then
\begin{enumerate}
\item If $X \subset \P^n_K$ is an irreducible subvariety such that
  $f^r(X) = X$ for some $r > 0$, then $X$ is a point or $X = \P^n_K$.
\item If $O_f(x)$ is infinite for $x \in \P^n(K)$ then $O_f(x)$
  is Zariski dense in $\P^n_K$.
\end{enumerate}
\end{thm}

\begin{proof}
  Since $f$ is generic, we may identify $f$ with $f_{n,d}$ defined
  above.  Any specialisation of $X$ to a subvariety of $\P^n$ defined
  over a finite extension of $k_{n,d}$ satisfies the same property as
  $X$, so we may assume that $X$ is defined over a finite extension of
  degree $m$ of $k_{n,d}$. Replacing $X$ by the union of its Galois
  conjugates, we may assume that $X$ is defined over $k_{n,d}$. Since
  $\dim(X) > 0$, it follows by \cite[Theorem 5.1]{nf-selfmaps} applied
  to $f^r$ that $X$ contains infinitely many periodic points of
  $f$. Since $X$ is defined over $k_{n,d}$, it then follows from Proposition
  \ref{prop:galois} that there is an infinite sequence of distinct
  integers $b_i$, $i = 1,2, \ldots$, such that $X$ contains all the
  periodic points of $f_{n,d}$ of period $b_i$ for all $i$. Let
  $g_{n,d}:\P^n_k \to \P^n_k$ be, as before, the map which raises each
  coordinate to its $d$-th power. By specialisation and Lemma
  \ref{lem:lift} (2), we obtain a subvariety $X'$ of $\P^n_k$ defined
  over $k$ which contains all the periodic points of $g_{n,d}$ of
  period $b_i$ for all $i$. 

  For $Z\subset \A^n_k$ any subvariety (with $k$ any field), if there
  is a sequence of finite subsets $S_i \subset \A^n(\ov{k})$ such that
  $(S_i)^n \subset Z(\ov{k})$ for all $i$ and $|S_i| \to \infty$ then
  $Z = \A^n_k$; this is well known and follows, for example, from
  \cite[Theorem 9.2]{tao-vu-ac}.  By the previous paragraph, for all
  $i$, $X'$ contains all elements of $\P^n(\ov{k})$ of the form
  $[\xi_0,\dots,\xi_{n-1},1]$ with $\xi_j \in
  \Per_{\phi_0}(b_i)$. Since $|\Per_{\phi_0}(b_i)| \to \infty$ with
  $i$ we must have $X' = \P^n_k$, hence $X = \P^n_{k_{n,d}}$, proving
  (1).

  To prove (2), we observe that the Zariski closure of $O_f(x)$ is
  mapped into itself by $f$. If it is infinite, it must contain a
  positive dimensional subvariety $X$ such that $f^r(X) = X$ for some
  $r>0$. By (1), we must have $X = \P^n_K$.
\end{proof}

\begin{rem}
  One can give a simpler proof of (1) using the Lefschetz trace formula
  rather than \cite[Theorem 5.1]{nf-selfmaps} if one only considers
  smooth $X$. Also, from Theorem \ref{thm:zhang} it follows that the
  elementary version of \cite[Theorem 5.1]{nf-selfmaps} asserting the
  density of preperiodic points suffices, but this introduces other
  dependencies.
\end{rem}

\section{The dynamical ``Manin--Mumford'' conjecture for generic
  endomorphisms}

\subsection{}

The original ``Manin--Mumford'' conjecture of Zhang asserted that for
any polarised endomorphism $f$ of a projective variety $X$ over a
field $K$ of characteristic zero, any subvariety $Y$ of $X$ containing
a Zariski dense set of preperiodic points is preperiodic. This was
known for abelian varieties and the multiplication by $m$ maps, but
was later shown to be false in general, even if $X = \P^n_K$, by
Ghioca and Tucker. Ghioca, Tucker and Zhang then proposed a modified
conjecture which takes into account the action of $f$ on the tangent
space; see the article \cite{ghioca-tucker-zhang} for a discussion of
the history, the statement of the modified version and some positive
results\footnote{It seems reasonable to expect the modified conjecture to
  hold even for quasi-polarised endomorphisms.}.

The following theorem implies that Zhang's conjecture, in its original
form, holds for generic endomorphisms of $\P^n_K$.

\begin{thm} \label{thm:zhang} 
  For $f: \P^n_K \to \P^n_K$ a generic endomorphism of degree $d > 1$
  over an algebraically closed field $K$ of characteristic zero, any
  infinite subset of $\P^n(K)$ consisting of $f$-preperiodic points is
  Zariski dense in $\P^n_K$.
\end{thm}

As in the other results of this paper, one of the key ingredients of
the proof is the Galois action on the set of periodic points. However,
Proposition \ref{prop:galois} alone does not suffice since we also
need the transitivity of the monodromy action on the set of
preperiodic points. It turns out that transitivity does not hold for
the monodromy action on $\prep_f(a,b)$, for $f$ as in Theorem
\ref{thm:morton} and $d>2$. Nevertheless, by considering a larger
family of polynomials, we prove in Proposition \ref{prop:prep} that
the monodromy action on $\prep_{f_{n,d}}(a,b)$ is indeed transitive
for all $a,b$. This then allows us to use a specialisation argument to
prove Theorem \ref{thm:zhang}.

\subsection{}

Fix $d>1$. For $b>0$, consider the polynomial $P_b(c) =
\phi_c^b(0)$. The roots of $P_b(c)$ are exactly the parameters $c$ so
that $0$ is a periodic point of period dividing $b$ for the polynomial
$\phi_c(z)$.
\begin{lem}[Gleason] \label{lem:gleason} 
  All roots of $P_b$ are multiplicity free.
\end{lem}
\begin{proof}
  The proof given in \cite[Lemma 19.1]{douady-hubbard-partie2} for
  $d=2$ goes through for general $d$ simply by replacing $2$ by $d$.
\end{proof}

\begin{lem} \label{lem:prep1} 
  Fix $d>1$. For $(c, \ep) \in \A^2$, let $\phi_{c,\ep}: \A^1 \to
  \A^1$ be given by $\phi_{c,\ep}(z) = z(z-\ep)^{d-1} + c$. Then the
  monodromy action on $\prep_{\phi_{c,\ep}}(1,b)$ is transitive for
  all $b>0$.
\end{lem}

\begin{proof}
  It follows from Theorem \ref{thm:morton} that the monodromy action
  on $\Per_{\phi_{c,\ep}}(b)$ is transitive for all $b>0$, so it
  suffices to prove that for any $b>0$ there exists $x \in
  \Per_{\phi_{c,\ep}}(b)$ and an element $\gamma$ of the monodromy so
  that $\gamma(x) = x$ and $\gamma$ cyclically permutes the $d-1$
  elements of $\phi_{c,\ep}^{-1}(x)$ (so if $d=2$ there is nothing to
  prove).

  By Lemma \ref{lem:gleason} and a simple counting argument, it
  follows that for any $b>0$, there exists $c_b \in \C$ so that $0 \in
  \Per_{\phi_{c_b}}(b)$. Since $0$ is a critical point of
  $\phi_{c_b}$ it is of multiplicity $1$ as an element of
  $\Per_{\phi_{c_b}}(b)$. It follows that for $|\ep| \ll 0$, there
  exists $c_{b,\ep}$ close to $c_b$ so that $0$ is a periodic point of
  $\Per_{\phi_{c_{b,\ep},\ep}}(b)$ of multiplicity $1$. By the
  definition of $\phi_{c,\ep}$, it then follows that $\ep$ is the
  unique element of $\prep_{\phi_{c_{b,\ep},\ep}}(1,b)$ which is
  mapped to $\phi_{c_{b,\ep},\ep}(0) = c_{b,\ep}$ by
  $\phi_{c_{b,\ep},\ep}$.

  We now consider the $1$-parameter family of maps $\phi_{c,\ep}$ with
  $\epsilon$ such that $|\ep| \ll 0$ fixed. In a neighbourhood of
  $c_{b,\ep}$, the element $c_{b,\ep} \in
  \Per_{\phi_{c_{b,\ep},\ep}}(b)$ deforms uniquely as an element of
  $\Per_{\phi_{c,\ep}}(b)$. However, since the critical points of
  $\phi_{c,\ep}$ are independent of $c$, it follows that the element
  $\ep \in \prep_{\phi_{c_{b,\ep},\ep}}(1,b)$ deforms to $d-1$
  distinct elements of $\prep_{\phi_{c,\ep}}(1,b)$ which are all
  mapped to the deformed periodic point above by $\phi_{c,\ep}$.

  We claim that the monodromy action of a small loop around
  $c_{b,\ep}$ gives us the required element $\gamma$. Since
  $c_{b,\ep}$ deforms uniquely as a periodic point, the mondromy
  action of $\gamma$ on this point is trivial as required. To prove
  that the second condition is satisfied let $C \subset \A^2$ be the
  curve consisting of all points $(z,c)$ so that $z$ is a preperiodic
  points of preperiod $1$ and period $b$ of $\phi_{c,\ep}$ (with $\ep$
  fixed as above). It suffices to prove that $C$ is smooth at the
  point $(\ep, c_{b,\ep})$.

  To see this we consider the explicit equation for $C$. It is given
  in an neighbourhood of $(\ep, c_{b,\ep})$ by
  \begin{equation} \label{eq:d-1} \phi_{c,\ep}^{ b+1}(z) -
    \phi_{c,\ep}(z) = 0 \ .
  \end{equation}
  To see that it is smooth at $(c_{b,\ep},\ep)$ it suffices to
  substitute $\ep$ for $z$ and check that the resulting polynomial in
  $c$, $\phi_{c,\ep}^{ b+1}(\ep)- \phi_{c,\ep}(\ep)$, has $c_{b,\ep}$
  as a root of multiplicity $1$. However, from the definition of
  $\phi_{c,\ep}$ it follows that
  \[
  \phi_{c,\ep}^{ b+1}(\ep)- \phi_{c,\ep}(\ep) = \phi_{c,\ep}^{ b+1}(0)
  - \phi_{c,\ep}(0) \
  \]
  so we may replace $\ep$ by $0$. In a neighbourhood of the point
  $(0,c_{b,\ep})$ the curve given by \eqref{eq:d-1} is smooth since it
  parametrizes periodic points of period $b$ and the periodic point
  $0$ of $\phi_{c_{b,e},\ep}$ is of multiplicity $1$. To show that the
  multiplicity of $c_{b,\ep}$ as a root of $\phi_{c,\ep}^{ b+1}(0) -
  \phi_{c,\ep}(0)$ is $1$ we may then specialize $\ep$ to $0$ so it
  suffices to consider the multiplicity of $c_b$ as a root of the
  polynomial $ P_b(c) = \phi_{c}^{ b}$ By Lemma \ref{lem:gleason},
  this multiplicty is indeed $1$ as required.
\end{proof}

To prove the transitivity of the monodromy action on
$\prep_{\phi_{c,\ep}}(a,b)$ for $a>1$, we shall need some results
about Misiurewicz points. We refer the reader to \cite{lau-schleicher}
and \cite{eberlein} for the basic facts that we use below, which
generalize results proved in \cite{douady-hubbard-partie1} in the case
$d=2$. Recall that $c_0 \in \C$ is called a Misiurewicz point if $c_0$
is a strictly preperiodic point of the map $\phi_{c_0}$. By the
results of \emph{op.~cit.}, for any strictly preperiodic angle $\theta
\in \Q/\Z$, there is a Misiurewicz point $c_0$ such that the parameter
ray with angle $\theta$ lands at $c_0$. By \cite[Lemma 8.3]{eberlein},
the preperiod of $\theta$ (with respect to multiplication by $d$) is
equal to the preperiod of $c_0$ (with respect to $\phi_{c_0}$) and the
period of the kneading sequence of $\theta$, $K(\theta)$, is equal to
the period of $c_0$ (with respect to $\phi_{c_0}$). For $\theta =
1/(d^a\cdot(d^b-1))$, $a,b > 0$, the preperiod of $\theta$ is $a$ and
the period of $K(\theta)$ is $b$, so there exists a Misiurewicz point
with any preperiod $a>0$ and period $b$.

A point $\l \in \C$ is called parabolic of period $b$ if it is the
landing point of a parameter ray with angle $\theta$ which is periodic
of period $b$. By the results of \cite{douady-hubbard-partie1},
\cite{eberlein} a parabolic point is never also a Misiurewicz point.

\smallskip 

To prove the transitivty of the monodromy action for $a>1$ we shall
need the following analogue of Lemma \ref{lem:gleason} due to Douady
and Hubbard \cite[Corollaire 8.5]{douady-hubbard-partie1}:

\begin{lem} \label{lem:dh} 
  For a Misiurewicz point $c_0$ as above, the equation
  $\phi_{c}^{a+1+b}(0) - \phi_{c}^{a+1}(0) = 0$ has a simple root at
  $c = c_0$.
\end{lem}

\begin{proof}
  The lemma is formulated and proved in \emph{op.~cit.}~only for
  $d=2$, however \cite[Proposition 8.5]{douady-hubbard-partie1} holds
  for general $d$ and so the proof of \cite[Corollaire
  8.5]{douady-hubbard-partie1} goes through if we substitute
  \cite[Theorems 8.1 \& 8.2]{eberlein} for \cite[Theorem
  8.2]{douady-hubbard-partie1}.
\end{proof}

\begin{lem} \label{lem:prep2} 
  The monodromy action of the two dimensional family of polynomials
  $\phi_{c,\ep}$ on $\prep_{\phi_{c,\ep}}(a,b)$ is transitive
  for all $a,b>0$.
\end{lem}

\begin{proof}
  We already know transitivity if $a = 1$. Thus, by induction we may
  assume $a>1$ and then it suffices to prove that for the one
  dimensional family of polynomials $\phi_c$, there exists $x \in
  \Per_{\phi_c}(a-1,b)$ and an element $\gamma$ of the monodromy such
  that $\gamma(x) = x$ and $\gamma$ induces a cyclic permutation on
  the $d$ elements of $\Per_{\phi_c}(a-1,b)$ comprising
  $\phi_c^{-1}(x)$.

  Since $0$ and $\infty$ are the only critical points of $\phi_{\l}$,
  the preperiodic points for general $\l$ are multiplicity free. Let
  $c$ be a Misiurewicz point of preperiod $a-1$ and period $b$ and
  consider a small loop $\gamma$ in the parameter plane around $c$,
  such that all parabolic points of period $b$ and all Misiurewicz
  points of preperiod $a$ and period $b$ are outside this loop. We
  note that the preperiodic points of $\phi_c$ of preperiod $\leq a$
  and period $b$ are multiplicity free since none of them are critical
  values and a Misiurewicz point is never a parabolic point. This
  remains true in a neighbourhood of $c$ so we may assume that this
  holds in a neighbourhood $U$ of $\gamma$ containing its interior; in
  particular, $c$ deforms uniquely as a preperiodic point $c_{\l} \in
  \prep_{\phi_{\l}}(a,b)$ as $\l$ varies in this neighbourhood.

  Since $\phi_c$ is totally ramified at $0$, $0$ is the unique element
  of $\prep_{\phi_c}(a,b)$ mapping to $c$ by $\phi_c$. By
  construction, for any other $\l \in U$, there are $d$ points of
  $\prep_{\phi_{\l}}(a,b)$ mapping to $c_{\l}$ and these points all
  come together at $0$ as $\l \to c$. The set of these points in a
  neighbourhood of $(c,0)$ is exactly the zero locus $D$ of the
  polynomial $\phi_{\l}^{a+1+b}(z) - \phi_{\l}^{a+1}(z)$. By Lemma
  \ref{lem:dh}, the multiplicity of this after setting $z = 0$ is $1$
  at $(c,0)$, so it follows that $D$ must be smooth at this
  point. Since the inverse image of $c$ in $D$ is a single point, it
  follows that the map induced by the projection to the first factor
  is totally ramified of degree $d$ at $(c,0)$. Consequently, the
  monodromy around $\gamma$ induces a cyclic permutation of order $d$
  on $\phi_{\l}^{-1}(c_{\l})$. We may thus take $x = c_{\l}$ for any
  $\l$ with $|\l| \ll 0$ to complete the proof.
\end{proof}

\begin{prop} \label{prop:prep} 
  The Galois action on $\prep_{f_{n,d}}(a,b)$ is transitive for all
  $a,b > 0$.
\end{prop}

\begin{proof}

  Since $f_{n,d}$ is defined over $k_{n,d}$, the Galois action on
  $\prep_{f_{n,d}}(a,b)$, $a,b\geq 1$, is compatible with the natural
  surjections
  \[
  \prep_{f_{n,d}}(a+1,b) \sr{f_{n,d}}{\lr} \prep_{f_{n,d}}(a,b) .
  \]
  Thus, by induction on $a$, it suffices to show that for any $a>0$,
  there exists an element $x \in \prep_{f_{n,d}}(a-1,b)$ such that for
  any $y,y' \in \prep_{f_{n,d}}(a,b)$ with $f_{n,d}(y) = f_{n,d}(y')$,
  there exists an element $\gamma$ of the monodromy such that
  $\gamma(y) = y'$.

  If $n=1$, the claim follows immediately from Lemma \ref{lem:prep2}
  so in the following we shall assume $n>1$.

  As before, we may assume that $k = \C$.  The proof of transitivity
  is similar to that for the case of periodic points, except that we
  replace the use of the maps $\phi_c$ by $\phi_{c,\ep}$. So for $0
  \leq i \leq n$ consider the family of endomorphisms $f_i: \P^n
  \times \A^{2n} \to \P^n \times \A^{2n}$ defined by
  \begin{multline*}
    f_i(([x_0,\dots,x_n],(c_1,\dots,c_n,\ep_1,\dots,\ep_n))) = \\
    ([x_0(x_0 -\ep_1x_i)^{d-1} + c_1x_i^d,\dots,x_{i-1}(x_{i-1} -
    \ep_ix_i)^{d-1}
    + c_ix_i^d,x_i^d,x_{i+1}(x_{i+1} - \ep_{i+1}x_i)^{d-1} + x_i^d, \\
    \dots, x_n(x_n - \ep_nx_i)^{d-1} + c_nx_i^d],
    (c_1,\dots,c_n,\ep_1,\dots,\ep_n)) \ .
  \end{multline*}
  On the open affine $U_i$ given by $x_i \neq 0$, $f_i$ is the product
  of the $n$ polynomials $\phi_{c_j,\ep_j}$.
 
Let $g_{n,d}$ be the $d$-power map as before. Contrary to the the case
of periodic points, the map from the locus of preperiodic points to
the base is not etale at preperiodic points of $g_{n,d}$ contained in
its ramification locus. However, the map is etale at preperiodic
points all of whose coordinates are non zero and this will suffice
(except when $d=2$) for our needs (\emph{cf.} the discussion in \S
\ref{subsec:mono}).

Suppose $a=1$. Let $\zeta$ be a primitive $d^b -1$-th root of unity
and let $x' \in \Per_{g_{n,d}}(b)$ be the point $[\zeta,1,\dots,1,1]$
and let $x = g_{n,d}(x')$. The preperiodic points $y$ such that
$g_{n,d}(y) = x$ are of the form $[\zeta\xi_1,\xi_2,\dots,\xi_{n},1]$,
where all the $\xi_i$ are $d$-th roots of unity and at least one of
them is not equal to $1$.

We now also assume that $d>2$. Using the monodromy action of the
family $f_n$ and Lemma \ref{lem:prep1} it follows that we may assume
that $\xi_i = \xi$, where $\xi$ is a fixed $d$-th root of unity or
$\xi_i = 1$. 

Since $d>1$, there exists a $d$-th root of unity $\xi'$ such that
$\xi' \neq 1,\xi$. Let $y = [\zeta\xi,1,\dots,1]$, $y' = [\zeta\xi,
\xi,\dots, \xi, 1,\dots,1]$, where there are $n_1$ $\xi$'s and $n_2$
$1$'s with $n_1, n_2>0$.  Each step of the following sequence of
transformations is given either by multiplying through by a constant
or by applying the monodromy of $f_i$ for some $i$ such that the
$i$-th coordinate is equal to $1$:
\begin{multline*}
y = [\zeta\xi,1,\dots,1] \to [1,\zeta^{-1}\xi^{-1},\dots,
\zeta^{-1}\xi^{-1}]
\to [1, \zeta^{-1}{\xi'}^{-1},\dots, \zeta^{-1}{\xi'}^{-1},
\zeta^{-1}{\xi}^{-1},\dots, \zeta^{-1}{\xi}^{-1}] \\
\to [\zeta \xi, \xi{\xi'}^{-1},\dots, \xi{\xi'}^{-1},1\dots,1] 
\to [\zeta  \xi, \xi, \dots, \xi,1,\dots,1] = y' \ .
\end{multline*}
In the last transformation we also use the fact that $\xi {\xi'}^{-1}$
is a $d$-th root of unity not equal to $1$ so, like $\xi$ and $\xi'$,
an element of $\prep_{\phi_{0}}(1,1)$.

Let $y'' = [\zeta,\xi,\dots,\xi,1\dots,1]$, where there are $n_1$
$\xi$'s and $n_2$ $1$'s with $n_1, n_2>0$ as before. Since $d>2$,
there exists $\zeta' \in \Per_{\phi_0}(b)$ such that
$\zeta\zeta'^{-1}$ also has period $b$. We then have a similar
sequence of transformations:
\begin{multline*}
y' =  [\zeta  \xi, \xi, \dots, \xi,1,\dots,1] \to
[1,\zeta^{-1}, \dots, \zeta^{-1},
\zeta^{-1}{\xi}^{-1}, \dots, \zeta^{-1}{\xi}^{-1}] \\
\to [1,\zeta'\zeta^{-1}, \dots, \zeta\zeta^{-1},
\zeta^{-1}{\xi}^{-1}, \dots, \zeta^{-1}{\xi}^{-1}]
\to [\zeta\xi, \zeta'\xi,\dots,\zeta'\xi,1,\dots,1] 
\to  [\zeta\xi, \zeta'\xi',\dots,\zeta'\xi',1,\dots,1] \\ \to
[\zeta\xi(\zeta'\xi')^{-1},1,\dots,1,(\zeta'\xi')^{-1},\dots,(\zeta'\xi')^{-1}]
\to
[\zeta(\zeta'\xi)^{-1},1,\dots,1,(\zeta'\xi)^{-1},\dots,(\zeta'\xi)^{-1}] \\
\to [\zeta, \zeta'\xi,\dots,\zeta'\xi,1,\dots,1] 
\to [1,\zeta^{-1}\zeta'\xi,\dots,\zeta^{-1}\zeta'\xi,\zeta^{-1},\dots,\zeta^{-1}]
\\ \to
[1,\zeta^{-1}\xi,\dots,\zeta^{-1}\xi,\zeta^{-1},\dots,\zeta^{-1}] 
\to [\zeta,\xi,\dots,\xi,1,\dots,1] = y'' \ . 
\end{multline*}
In the above we have used that $\zeta,\zeta', \zeta\zeta'^{-1}$ and
their inverses are in $\Per_{\phi_0}(b)$ and each one of these
multiplied by $\xi, \xi', \xi\xi^{-1}$ or any of their inverses is an
element of $\prep_{\phi_0}(1,b)$. By symmetry, we then conclude that
the desired transitivity holds in this case.

Now suppose $d=2$. In this case $\prep_{\phi_0}(1,1) = \{-1\}$, a
singleton, so the above argument breaks down and we will need to
consider paths passing through elements of $\prep_{g_{n,d}}(1,b)$, one
of whose coordinates is $0$. This is justified by Lemma \ref{lem:deg2}
below.

As before, let $y = [-\zeta,\dots,-\zeta,1]$ with $\zeta$ a primitive
$2^b-1$-th root of unity and $y' =
[-\zeta,\dots,-\zeta,\zeta,\dots,\zeta,1]$, where there is at least
one $-\zeta$ and one $\zeta$. 
We need to consider the cases $b=1$ and $b>1$ separately.

First suppose $b=1$, so $\zeta = 1$.  We then have the sequence of
transformations
\[
y =  [-1,\dots, -1,1] \to [1,\dots,1,-1] \to [1,\dots,1,0,-1] \\
\to [-1,\dots,-1,0,1] \to [-1,\dots,-1,1,1] \ .
\]
Repeating this procedure we see that we can connect $y$ to $y'$ and by
symmetry the transitivity follows in this case.

Now suppose $b >1$, so there exists $\zeta' \in \Per_{\phi_0}(b)$ so
that, as before, $\zeta \zeta'^{-1} \in \Per_{\phi_0}(b)$. We then
have a sequence of transformations
\begin{multline*}
y = [-\zeta,\dots,-\zeta,1] \to [1,\dots,1,-\zeta^{-1}] \to
[1,\dots,1,0,-\zeta^{-1}] 
\to [-\zeta, \dots,-\zeta,0,1] \\\to  [-\zeta, \dots,-\zeta,1,1] \to
[1,\dots,1,-\zeta^{-1},-\zeta^{-1}] 
\to [1,\dots,1,-\zeta'^{-1},-\zeta^{-1}] \\ \to
[-\zeta^{-1},\dots,-\zeta^{-1},\zeta\zeta'^{-1},1] 
\to [-\zeta^{-1},\dots,-\zeta^{-1},\zeta,1] \ .
\end{multline*}
Repeating this procedure, we see that we can connect $y$ to $y'$ and
then by symmetry transitivity follows.

Finally, suppose $a>1$. By induction, we can choose $x$ to be an arbitrary
point of preperiod $a-1$ and period $b$, so we let $x' =
[\zeta,\zeta,\dots,\zeta,1]$ where $\zeta$ is an element of
$\prep_{\phi_0}(a,1)$ and $x = g_{n,d}(x)$. The points in
$g_{n,d}^{-1}(x)$ are of the form $[\zeta\xi_1,\dots,\zeta\xi_n,1]$
where $\xi_1$ is a $d$-th root of $1$, so $\zeta\xi_i \in
\prep_{\phi_0}(a,1)$. One sees that the monodromy acts transitively on
$g_{n,d}^{-1}(x)$ simply by considering the monodromy of the family of
maps $f_n$ and applying Lemma \ref{lem:prep2}.

\end{proof}

Let $F:\P^n_k \times \Mor_{n,d/k} \to \P^n_k \to \Mor_{n,d/k}$ be the
universal morphism of degree $d$ and let $P_F(b+1,b) \subset
\Mor_{n,d/k}$ be the subscheme defined in \S \ref{subsec:p}.  The
fibre of the projection map from $P_F(b+1,b)$ to $\Mor_{n,d/k}$ over
any point $f \in \Mor_{n,d/k}$ consists of $f$-preperiodic points of
preperiod at most $1$ and period dividing $b$.
\begin{lem} \label{lem:deg2} 
  If $d=2$ and $\ch(k) \neq 2$, then $P_F(b+1,b)$ is smooth at any
  preperiodic point of $g_{n,2}$ of preperiod $1$ and period $b$.
\end{lem}
\begin{proof}
  We have $F(P_F(b+1,b)) = P_F(b)$. As we have already seen, $ P_F(b)$
  is smooth at all periodic points of $g_{n,d}$. Moreover, $F$ is
  analytically locally at any point of $\prep_{g_{n,d}}(1,b)$ with
  exactly one coordinate equal to $0$, a cyclic cover of degree
  $2$. Thus to prove smoothness it suffices to show that the
  discriminant of $F$ intersects $P_F(b)$ transversely at any point of
  $\Per_{g_{n,d}}(b)$ with exactly one coordinate equal to $0$.

  To prove transversality, it then suffices to restrict to any
  subvariety of $\Mor_{n,d}$ and prove transversality for the induced
  subvarieties. By considering, say, the family $f_i$ as in the proof
  of Proposition \ref{prop:galois}, where $i$ is such that the $i$-th
  coordinate of the point under consideration is nonzero, we reduce to
  the case of the one parameter $\phi_c$, $c \in \A^1_c$ and we need
  to prove transversality at the point $(z,c) = (0,0)$. The
  discriminant locus is given by $z=0$ (since $d=2$) and the locus of
  fixed points by $z^2 +c - z = 0$, so the lemma follows.
\end{proof}

\subsection{}

\begin{proof}[Proof of Theorem \ref{thm:zhang}]
  Let $X$ be the Zariski closure of an infinite subset of preperiodic
  points.  By Lemma \ref{lem:lift} (1), all preperiodic points of a
  quasi-polarised map are defined over the algebraic closure of the
  base field, so we may assume without loss of generality that $K$ is
  the algebraic closure of $k_{n,d}$ and $X$ is defined over a finite
  extension of $k_{n,d}$. By replacing $X$ by the union of its Galois
  conjugates, we may then assume that $X$ is defined over $k_{n,d}$.

  Since $\prep_{f_{n,d}}(a,b)$ is finite for all $a,b$, there exists
  an infinite sequence of tuples $(a_i,b_i)$, with $a_i\geq 0$ and
  $b_i>0$, so that $X$ contains a point $x_i \in
  \prep_{f_{n,d}}(a_i,b_i)$ for all $i$. Since $X$ is defined over
  $k_{n,d}$ it follows from Propositions \ref{prop:galois} and
  \ref{prop:prep} that $\prep_{f_{n,d}}(a_i,b_i) \subset X$ for all
  $i$.  We let $X'$ be the Zariski closure of the specialisations,
  over the point in $\Mor_{n,d}$ corresponding to $g_{n,d}$, of the
  preperiodic points in $X(K)$.  Since all preperiodic points lift to
  the generic fibre by Lemma \ref{lem:lift}, it follows that $X'
  \subset \P^n_k$ has the same properties as $X$ but with respect to
  $\Per_{g_{n,d}}$.

  For any $i$, the set of points in $\P^n_{\ov{k}}$ of the form
  $[\xi_0,\dots,\xi_{n_1},1]$ with $\xi_i \in \prep_{\phi_0}(a_i,b_i)$ is
  contained in $\prep_{g_{n,d}}(a_i,b_i)$, hence in $X'$. As in the
  proof of Theorem \ref{thm:inv}, it then follows that $X' = \P^n_k$,
  so $X = \P^n_{k_{n,d}}$.

\end{proof}

\begin{rem} \label{rem:nongen} 
  Note that we do not use the full strength of the genericity
  hypothesis in the proofs of this section or of the previous one. It
  suffices to assume that the morphism under consideration corresponds
  to the generic point of an irreducibility subvariety of
  $\Mor_{d,n/k}$ which contains all the families $f_{c,\epsilon}$ and
  is smooth at $g_{n,d}$. Since all the $f_{c,\ep}$ are smooth and
  have dimension $2n$, there exist such subvarieties for all $n$ with
  dimension independent of $d$.
\end{rem}

\section{The dynamical ``Mordell--Lang'' conjecture for generic
  endomorphisms}

\subsection{}

Let $(X,f)$ be an algebraic dynamical system over a field $K$ of
characteristic zero, \emph{i.e.}, $X$ is an algebraic variety and
$f:X\to X$ is a morphism. The conjecture of Ghioca and Tucker
\cite{ghioca-tucker-conj} asserts that if $x \in X(K)$ and $Y$ a
subvariety of $X$ are such that $O_f(x) \cap Y(K)$ is infinite then
there is a periodic subvariety $Z$ of $X$ with $Z \subset Y$ and $Z(K)
\cap O_f(x) \neq \emptyset$. It has been proved when $f$ is etale by
Bell, Ghioca and Tucker in \cite{bell-ghioca-tucker} and in a few
other cases. It is not known in general if $X = \P^n_K$ and $\deg(f) >
1$; this was the original case investigated by Denis
\cite{denis-suites} who proved the assertion under the assumption that
$O(x) \cap Y(K)$ is large in a suitable sense.

For $(X,f) = (\P^n_K, f)$, with $f$ a generic endomorphism, by Theorem
\ref{thm:inv} there are no non-trivial $f$-periodic subvarieties
contained in $\P^n_K$, so the conjecture in this case is equivalent to
the following:

\begin{thm} \label{thm:gt} 
  Let $f: \P^n_K \to \P^n_K$ be a generic endomorphism of degree $d >
  1$ over an algebraically closed field $K$ of characteristic zero.
  For each $x \in \P^n(K)$, every infinite subset of $O_f(x)$ is
  Zariski dense in $\P^n_K$.
\end{thm}

The idea of the proof is as follows: we first use specialisation to
reduce to the case that $K$ is a finite extension of $k_{n,d}$. We
then show using a $p$-adic argument, for a prime $p$ dividing $d$,
that any $Y$ such that $O_f(x) \cap Y$ is infinite must contain
infinitely many periodic points, or there exists a prime $q$ not
diving $d$, such that $x, f$ have specialisations $\ov{x}, \ov{f}$
over $\F_q$ with $\ov{f}$ etale on the orbit of $\ov{x}$. Both these
conditions lead to the conclusion that $Y = \P^n_K$; the first from
Theorem \ref{thm:zhang} and second by using the results of
\cite{bell-ghioca-tucker}, which we recall as Lemma \ref{lem:bgt}
below.

\subsection{}

The following lemma is an immediate consequence of the results of
\cite{bell-ghioca-tucker} if $L = \Q_p$. For general $L$, it follows
from the methods of \cite{bell-ghioca-tucker} if one replaces Theorem
3.3 therein by Theorem 7 of \cite{amerik-non-preperiodic}.

\begin{lem}\label{lem:bgt} 
  Let $L/\Q_p$ be a finite extension, $\pi:X \to \Spec(R)$ a smooth
  scheme of finite type over the ring of integers $R$ of $L$, and $f:X
  \to X$ a morphism over $\Spec(R)$. Suppose $x \in X(R)$ is such that
  $f$ is etale on the orbit of $x$. If $Y \subset X$ is any closed
  subscheme with $Y \cap O_f(x)$ infinite, then $Y_L$ contains a
  positive dimensional periodic subvariety of $X_L$.
\end{lem}

\begin{lem} \label{lem:fin} 
  Theorem \ref{thm:gt} for arbitrary extensions $K$ of $k_{n,d}$
  follows from the case of finite extensions.
\end{lem}

\begin{proof}
  Without loss of generality we may assume that $k = \Q$ and $K$ is a
  finitely generated extension of $k_{n,d}$. Let $x \in \P^n(K)$ and
  let $Y$ be a subvariety of $\P^n_K$ such that $O_f(x) \cap
  Y(K)$ is infinite.

  Since $K$ is finitely generated, there exists a smooth irreducible
  scheme $M$ of finite type over $\Z$ with function field $K$ and a
  dominant morphism $\pi:M \to \Mor_{n,d}$ inducing the inclusion
  $k_{n,d} \subset K$ on function fields. Let $\mf{f}: M \times \P^n
  \to M \times \P^n$ be the pullback of the universal morphism from
  $\Mor_{n,d} \times \P^n$, so $\mf{f}$ restricted to the generic
  fibre of the first projection is equal to $f$.  Let $\mf{x}$ (resp.~
  $\mf{Y}$) be the Zariski closure of $x$ (resp.~$Y$) in $ M \times
  \P^n$. By shrinking $M$ if necessary, we may assume that $\mf{x}$
  and $\mf{Y}$ are flat over $M$.

  Since $p$ is a dominant finite type morphism, there exists a point
  $f'$ of $M$ (which we think of as an endomorphism of $\P^n$ using
  $\pi$) mapping to the generic point of $\Mor_{n,d}$ and so that the
  residue field $K'$ at $f'$ is a finite extension of $k_{n,d}$. Let
  $x'$ (resp.~ $Y'$) denote the fibre of $\mf{x}$ (resp.~$\mf{Y}$) over
  $f'$. If $O_{f'}(x')$ is infinite, then so is $O_{f'}(x') \cap Y'$
  hence by the condition on $K'$ it would follow that $Y' = \P^n_{K'}$
  which (by flatness) implies $Y = \P^n_K$.

  If $O_{f'}(x')$ is finite, then $x'$ is $f'$-preperiodic. Since $f'$
  is generic, it follows from Corollary \ref{cor:ram} that
  $O_{f'}(x')$ does not intersect the ramification locus of $f'$. Let
  $|O_{f'}(x')| = n$ and let $\mf{Z}$ denote the Zariski closure of
  $\{x, f(x),\dots,f^{n-1}(x)\}$ in $M \times \P^n$. Let
  $\mf{R}\subset M \times \P^n$ be the ramification locus of $\mf{f}$
  and consider the closed subset $\mf{R} \cap \mf{Z}$ of $M \times
  \P^n$. By the above, the fibre of this subset over $f'$ is empty so,
  by the properness of $\P^n$, its projection in $M$ is a proper
  closed subset. Replacing $M$ by the complement of this subset we may
  assume that $\mf{R} \cap \mf{Z} = \emptyset$.

  Now let $f''$ be any closed point (which we again think of as an
  endomorphism) of $M$ which lies in the closure of $f'$. Since $M$ is
  of finite type over $\Z$, the residue field of $f''$ is a finite
  field $F$. Let $x''$ (resp. $Y''$) denote the fibre of $\mf{x}$
  (resp. $\mf{Y}$) over $f''$. Since $f''$ is in the closure of $f'$,
  $x''$ is in the closure of $x'$ hence $|O_{f''}(x'')| \leq n$. Since
  $\mf{R} \cap \mf{Z} = \emptyset$, it follows that $f''$ is
  unramified at all points of $O_{f''}(x'')$. Let $W(F)$ be the ring
  of Witt vectors of $F$. Since $M$ is smooth over $\Z$, by Hensel's
  lemma the set of points in $M(W(F))$ which specialise to $f''$ is in
  bijection (after choosing local coordinates) with $W(F)^n$. The
  subset consisting of points which lie in a proper closed subscheme
  of $M$ is a countable union of nowhere dense (in the adic topology)
  subsets. It follows by Baire's theorem that there exists a point in
  $M(W(F))$ specialising to $f''$ and not lying in any proper closed
  subscheme. Letting $L$ be the quotient field of $W(K)$, it follows
  that the image of the induced map from $\Spec(L)$ to $M$ must be the
  generic point. We thus get an inclusion of $K$ into $L$ and we may
  apply Lemma \ref{lem:bgt} with $R = W(K)$ and $X = \Spec(W(F))
  \times \P^n$ to the base change of $\mf{f}$, $\mf{x}$, $\mf{Y}$ via
  the morphism $\Spec(W(F)) \to M$ to conclude using Theorem
  \ref{thm:inv} (1) that $Y = \P^n_K$.

\end{proof}

\begin{lem} \label{lem:power} 
  Let $p$ be a prime and $g_{n,d,p}$ denote the endomorphism of
  $\P^n_{\F_p}$ given by raising each coordinate to its $d$-th power.
  Let $X \subset \P^n_{\ov{\F}_p}$ be a positive dimensional
  subvariety. Then the set $\cup_{r\geq 0}\  g_{n,d,p}^r(X(\ov{\F}_p))$
  contains periodic points of infinitely many distinct periods.
\end{lem}

\begin{proof}
  Since $g_{n,d,p}$ preserves the standard decomposition of $\P^n$ as
  a disjoint union of affine spaces, by projecting to a suitable
  coordinate, we reduce to the statement for $n=1$ in which case the
  statement is obvious.
\end{proof}

\begin{rem}
  We expect that the lemma holds with $g_{n,d,p}$ replaced by an
  arbitrary quasi-polarised morphism---or even more generally with
  some extra conditions on $X$---defined over a finite field, but this
  seems much harder to prove. However, for endomorphisms of abelian
  varieties the corresponding statement can indeed be proved.
\end{rem}

\begin{lem} \label{lem:unbounded} 
  Let $d>1$ be an integer and $p$ a prime such that $p \mid d$. Let
  $\phi$ be the morphism $\A^2 \to \A^2$ given by $\phi(x,c) = (x^d +
  c,c)$ over the field $\ov{\F}_p$. Then
\begin{enumerate}[(a)]
\item For $c \in \ov{\F}_p$, the monodromy action on the set of fixed
  points of $\phi_c$ is transitive.
\item Let $X \subset \A^2$ be an irreducible subvariety of dimension
  one, mapping dominantly to $\A^1$ via the second projection
  $p_2$. Assume that the intersection of $X$ with the generic fibre of
  $p_2$ is not a preperiodic point of $\phi$.  Then the $\phi$-periods
  of the points in $X(\ov{\F}_p)$ (which are all preperiodic) are
  unbounded.
\end{enumerate}
\end{lem} 

\begin{proof}
  We will use elementary intersection theory on $\P^1 \times Y$, with
  $Y$ a smooth projective curve. 

  Let $X_n = \phi^n(X)$ for $n \geq 0$. By replacing $X$ by
  $\phi^r(X)$ for some $r \geq 0$, we assume that the map $p_2:X_n \to
  \P^1$ has degree $e$ for all $n$.

  For any integer $b>0$, let $P_b$ be the locus of points in $\A^2$
  which are $\phi$-periodic of period $b$. Since $p\mid d$, $\phi$
  is inseparable on the fibres of $p_2$, so the graph of
  $\phi^b|_{\P^1 \times \{c\}}$ intersects the diagonal in $\A^2$
  transversely for all $c \in \A^1$ and all $b$. It follows that $P_b$
  is a finite etale cover of $\A^1$ via $p_2$ and $P_b \cap P_{b'} =
  \emptyset$ for $b \neq b'$.

  Let $\ov{P}_b$ be the closure of $P_b$ in $\P^1 \times \P^1 \supset
  \A^1 \times \A^1 = \A^2$ and let $\ov{X}_n$ be the closure of $X_n$
  in $\P^1 \times \P^1$.  The curve $P_b$ is a subcurve of the curve
  $Q_b$ in $\A^2$ with equation
\[
(\dots((x^d +c)^d + c)^d\dots)^d + c - x = 0\ ,
\]
where we have $b$ pairs of brackets. Replacing $c$ by $1/c'$ in the
above equation and multiplying through by ${c'}^{d^{(b-1)}}$ we get
the equation
\[
(\dots((c'x^d + 1)^d + {c'}^{d^{b-1} -d^{b-2}})^d  \dots)^d +
{c'}^{d^{b-1} - 1} - {c'}^{d^{b-1}}x = 0\ .
\] 
It follows that the only point on all the $\ov{P}_b$ intersecting the
fibre over $c = \infty$ is the point at infinity on this fibre and the
support of $\ov{P}_b \cap \ov{P}_{b'}$ is equal to this point if $b
\neq b'$. When $b=1$, the equation is
\[
c'x^d + 1 -c'x = 0 \ .
\]
One then sees that $\ov{P}_1$ is irreducible since the equation shows
that it is smooth at the point $(\infty,\infty)$ and the closure in
$\P^1 \times \P^1$ of any irreducible component of $P_1$ must contain
this point; this proves (a).

If the $\phi$-periods of all points in $X(\ov{\F}_p)$
are bounded,   then there must exist some $b>0$ so that
$|\ov{X}_n \cap \ov{P}_b| \to \infty$ as $n \to \infty$ and, by
assumption, $\ov{X}_n \nsubseteq \ov{P}_b$ for all $b$. Writing
$[\ov{X}_n] = e[\{0\} \times \P^1] + a_n[\P^1 \times \{0\}]$ in
$\ns(\P^1 \times \P^1)$, it follows that $a_n \to \infty$ as $n \to
\infty$.

Let $\ov{X}'$ be the normalisation of $\ov{X}$ and denote by $\eta:
\ov{X}' \to \P^1$ the composition of the normalisation map and the
projection $p_2$. Let $x_1,x_2,\dots,x_r$ be the points in $\ov{X}$
mapping to the point $(\infty,\infty)$ in $\P^1 \times \P^1$. If
$r=0$, it follows that ${X} \cap P_b \neq \emptyset$ for all $b$ so we
may assume that $r>0$. Let $b_1,b_2,\dots,b_{r+1}$ be any distinct
integers and let $\gamma: Y \to \ov{X}'$ be a Galois cover such that
there are components $S_i$ of $(\mr{id} \times \eta
\gamma)^{-1}(\ov{P}_{b_i})$ in $\P^1 \times Y$ which map
isomorphically to $Y$ via the second projection. Let $Z$ be the
section of this projection induced by the tautological section of
$\P^1 \times \ov{X}'$ and let $Z_n = \psi^n(Z)$, where $\psi$ is the
map of $\P^1 \times Y$ induced by $\phi$. We may write $[S_i] = [\ast
\times Y] + s_i[\P^1 \times \ast]$ and $[Z_n] = [\ast \times Y] +
a_n[\P^1 \times \ast]$ in $\ns(\P^1 \times Y)$ where $s_i \geq 0$ and
$\ast$ denotes any point. It follows that the intersection number $S_i
\cdot Z_n = s_i + a_n \to \infty$ as $n \to \infty$ for all
$i=1,2,\dots,r$.

Now consider the local intersection multiplicity $I_y(S_i,Z_n)$ of
$S_i$ and $Z_n$ at a point $\infty \times y$, where $y \in Y$ is such
that $\gamma(y) = x_j$ for some $j$. If this is bounded for all such
$y$ and all $n$, since $S_i \cdot Z_n \to \infty$, it would follow
that for large $n$, $S_i$ and $Z_n$ must intersect in a point $(z,y')$
such that $\eta \gamma(y') \in \A^1 \subset \P^1$ which implies that
$X_n \cap P_{b_i} \neq \emptyset$.

Suppose this is not the case, so $I_y(S_i,Z_{n_l}) \to \infty $ for
some infinite sequence $n_l \to \infty$. Since the $S_i$ are all
distinct smooth curves and there are only finitely many of them, it
follows that $I_y(T,Z_{n_l})$ must remain bounded as $n_l \to \infty$
where $T$ runs over all $S_{i'}$ for $i' \neq i$, and all of their
Galois conjugates. Up to Galois conjugation there are only $r$ points
$y$ as above, so it follows that we must have that for all large
$n_l$, there exists $i_{n_l} \in \{1,2,\dots,r+1\}$ so that
$I_y(S_{i_{n_l}},Z_{n_l})$, is bounded by an integer independent of
$n_l$ for all $y$ as above. It follows that we must have $X_{n_l} \cap
P_{b_{i_{n_l}}} \neq \emptyset$.

By choosing infinitely many disjoint sets of $r+1$ distinct integers
$\{b_1,b_2,\dots,b_{r+1}\}$ as above, we see that $X_{n_b} \cap P_b
\neq \emptyset$ for infinitely many distinct integers $b$ (and $n_b$
depending on $b$). Since all the $P_b$ are disjoint, it follows that
$X$ contains preperiodic points of infinitely many distinct periods.
\end{proof}

\begin{rem}
  We also expect this lemma to hold in much greater generality,
  \emph{e.g.}, for any one parameter family of maps defined over
  $\ov{\F}_p$.
\end{rem}

The following lemma is the key to our construction of a periodic point
in $Y$ under the assumption that $O(x) \cap Y(K)$ is infinite.

\begin{lem} \label{lem:invlim} 
  Let $L/\Q_p$ be a finite extension, $\pi:X \to \Spec(R)$ a smooth
  projective scheme over the ring of integers $R$ of $L$, and $f:X \to
  X$ a quasi-polarised morphism over $\Spec(R)$.  Assume that the
  differential of $f$, $df$, is $0$ on the special fibre of $X$. For
  any $x \in X(L) = X(R)$, let $b$ be the period of the reduction
  $\ov{x}$ of $x$ in the special fibre of $X$. Then for any integer $a
  \geq 0$, the sequence of points $f^{a + bn}(x)$ converges to a
  periodic point of $X(L)$ of period $b$.
\end{lem}

\begin{proof}
  Replacing $f$ by $f^n$ and $x$ by $f^{a'}(x)$, for any integer $a'$
  greater than the preperiod of $\ov{x}$, we may assume that $\ov{x}$
  is a fixed point of $f$ and we then need to prove that $f^n(x)$
  converges to a fixed point.

  Since $f$ is quasi-polarised, by Lemma \ref{lem:lift}
  (2), $\ov{x}$ lifts to a fixed point $y$ of $f$ defined over a finite
  extension of $L$; by replacing $L$ by this extension, we may assume
  that $y \in X(L)$.

  Let $A$ be the completion of the local ring of $\ov{x}$ on
  $X$. Since $\pi$ is smooth, $A \cong R[[z_1,z_2,\dots,z_n]]$, where
  $n+1 =\dim(X)$. Using any such isomorphism, the set of points of
  $X(L)$ which specialize to $\ov{x}$ is identified with the set
  $(m_R)^n$, where $m_R$ is the maximal ideal of $R$. We fix such an
  isomorphism which we also assume identifies $y$ with $(0,\dots,0)
  \in (m_R)^n$.

  Since $\ov{x}$ is a fixed point of $f$, $f$ induces an endomorphism
  of $A$ which, with respect to the chosen isomorphism, is given by an
  $n$-tuple of elements $(f_1,f_2,\dots,f_n)$ in the maximal ideal of
  the local ring $R[[z_1,z_2,\dots,z_n]]$. Moreover, since $f$ fixes
  $y$, it follows that the constant term of each $f_i$ is $0$. Since
  $df$ is assumed to be zero on the special fibre of $X$, it follows
  that the coefficients of the linear term of each $f_i$ lies in
  $m_R$.  For any $\l = (\l_1,\l_2,\dots,\l_n) \in (m_R)^n$, let $|\l|
  = \max_i \{|\l_i|\}$. The conditions on the $f_i$ imply that for any
  such $\l$, $|f(\l)| < |\l|$ if $\l \neq (0,\dots,0)$. Since $R$ is a
  discrete valuation ring, it follows that for any such $\l$ we have
  that $f^n(\l) \to (0,0,\dots,0)$ as $n \to \infty$, hence $f^n(x)
  \to y$ as $n \to \infty$.
\end{proof}

\subsection{}

\begin{proof}[Proof of Theorem \ref{thm:gt}]
  By Lemma \ref{lem:fin}, we may assume that $K$ is a finite extension
  of $k_{n,d}$. Furthermore, we may assume without loss of generality
  that our base field $k=\Q$.
 
  Let $x \in \P^n(K)$ and assume that $Y$ is a subvariety defined over
  $K$ such that $I = O(x) \cap Y(K)$ is infinite. Let $X$ be the
  Zariski closure of the image of $x$ in $\Mor_{n,d} \times \P^n_{\Z}$
  and let $\chi$ denote the map $X \to \Mor_{n,d}$ induced by
  projection to the first factor. 

  Let $p$ be a prime dividing $d$. Since $\Mor_{n,d}$ is smooth over
  $\Spec(\Z)$, there is a map $g: \Spec(\Z_p) \to \Mor_{n,d}$ such
  that the generic point of $\Spec(\Z_p)$ maps to the generic point of
  $\Mor_{n,d}$ and the closed point maps to the point correspoding to
  $g_{d,n,p}$, the $d$-power map over $\F_p$.  Since $p \mid d$, the
  differential of the endomorphism of $\P^n_{\Z_p}$ corresponding to
  $g$, which we also denote by $g$, is zero on the special fibre.
  Suppose the fibre $X_{g_{d,n,p}}$ of $\chi$ over $g_{d,n,p}$ is
  infinite. By Lemma \ref{lem:power}, the set $\cup_{r\geq
    0}g_{d,n,p}^r(X_{g_{d,n,p}}(\ov{F}_p))$ contains infinitely many
  periodic points. By applying Lemma \ref{lem:invlim}, we can lift all
  these periodic points to periodic points of $f_{n,d}$ contained in
  $Y$. It then follows from Theorem \ref{thm:zhang} that $Y = \P^n_K$.
  Thus, we may assume from now on that $\chi$ is finite over an open
  neighbourhood of $g_{n,d,p}$.

  By replacing $x$ by $f^r(x)$ for some large $r$, we may assume that
  $X_{g_{d,n,p}}$ contains a periodic point $x' =
  [x_0',x_1',\dots,x_n']$ with $x_i' \in \ov{\F}_p$. Since $\chi$ is
  finite in a neighbourhood of $g_{n,d,p}$, $\Mor_{d,n}$ is smooth,
  hence normal, and $X$ is irreducible, it follows from the going down
  theorem that if none of the $x_i' = 0$ then the fibre $X_{g_{n,d}}$
  of $\chi$ over $g_{n,d}$ contains a point $\tilde{x}'$ lifting
  $x'$. By specialisation, it follows that for all large primes $q$
  the fibre of $\chi$ over $g_{n,d,q}$ contains a point all of whose
  coordinates are non-zero or, equivalently, not contained in the
  ramification locus of $g_{n,d,q}$. Since this locus is invariant
  under $g_{n,d,q}$, we may apply Lemma \ref{lem:bgt} conclude the
  existence of a positive dimensional periodic subvariety of $Y$
  which, by Theorem \ref{thm:inv}, implies $Y = \P^n_K$.

  We now use Lemma \ref{lem:unbounded} to show that such an $x'$ must
  exist, at least after replacing $x$ by a Galois conjugate, or $Y$
  must contain infinitely many periodic points, both cases leading to
  the conclusion that $Y = \P^n_K$. Let $x'$ be as above and suppose
  that $x_0' = 0$. Some other coordinate must be nonzero, so by
  symmetry we may  assume that $x_n' \neq 0$, and then by multipliying
  through by a scalar we may assume $x_n' = 1$. Consider the family of
  endomorphisms $\psi_c$ of $\P^n_{\F_p}$ parametrised by $\A^1$ given
  by
  \[
  \psi_c([x_0,x_1,\dots,x_n]) = [x_0^d + cx_n^d,x_1^d,\dots,x_n^d] \ ,
  \]
  so $\psi_0 = g_{n,d,p}$. Note that on the affine space given by the
  locus with $x_n \neq 0$, $\psi_c$ is given in affine coordinates by 
\[
(z_0,z_1,\dots,z_{n-1}) \mapsto (z_0^d + c,z_1^d,\dots,z_{n-1}^d) \ .
\]

Let $S \subset \Mor_{n,d}$ be the subscheme corresponding to the
family $\psi_c$. By the going down theorem there is an irreducible
component $T$ of $\chi^{-1}(S)$ mapping onto $S$ and containing the
point $(g_{n,d,p},x')$. Let $T'$ be the image of $T$ in $\P^n_{\F_p}$
under the projection of $X$ to $\P^n$. Since $x_n' = 1$, $T'$ is not
contained in the locus given by $x_n = 0$, so by projecting to the
first $n$ coordinates we get a rational map $\rho$ from $T$ to $\A^n$.

  Suppose the composition of $\rho$ with the $i$-th projection is
  non-constant for some $i$, $0 < i \leq n-1$. Since the action of
  $\psi_c$ on the $i$-th coordinate doesn't depend on $c$, it follows
  that $T(\ov{\F}_p)$ must contain preperiodic points of arbitrarily
  large period. By Lemma \ref{lem:invlim} as before, we obtain
  infinitely many periodic points in $Y$, forcing $Y = \P^n_K$.

  So suppose $\rho$ composed with all the $i$-th projections are
  constant for $i>0$ and let $\sigma: T \to \A^1 \times S $ be given
  by $( \pi_0 \rho, \chi)$. By applying Lemma \ref{lem:unbounded} (b),
  it follows that if the image of $T$ is not contained inside a
  preperiodic curve for the map $\phi$ (using the identification of
  $S$ with $\A^1)$ there must be $\phi$ preperiodic points in the
  image with unbounded period. By the construction of $\psi_d$, it
  follows that there are preperiodic points on $T$ of unbounded
  period. As before, this implies that $Y = \P^n$.

  The last case we need to consider is when the image of $T$ lies in a
  preperiodic curve. By replacing $x$ by an element in its orbit if
  necessary, we may assume that this image lies in the periodic
  locus. Now $0$ is a fixed point of the map $z \mapsto z^d$ and the
  point $(0,0)$ is contained in the image of $T$ by construction. By
  Lemma \ref{lem:unbounded} (a), it follows that the point $(0,1)$ is
  also in the image of $T$. We conclude that $X_{g_{d,n,p}}$ contains
  the periodic point $x'' = [1, x_1',\dots,x_{n-1}',1]$. By replacing
  $x'$ with $x''$ and repeating the above argument if necessary, we
  conclude that $Y$ contains infinitely many periodic points in which
  case it must be $\P^n_k$, or $X_{g_{d,n,p}}$ contains a periodic
  point $x' = [x_0',x_1',\dots,x_n']$ with $x_i' \neq 0$ for all $i$.
  As we have already seen, this also implies that $Y = \P^n_K$,
  concluding the proof.
\end{proof}

\begin{rem}
  Note that a similar statement to Remark \ref{rem:nongen} holds: it
  suffices to consider generic points of irreducible subschemes of
  $\Mor_{d,n}$ which contain all the families $f_i$ and and are smooth
  at the point $g_{n,d,p}$ for some prime $p$ dividing $d$.
\end{rem}

\def\cprime{$'$}


\end{document}